\documentclass[12pt,oneside,reqno]{amsart}
\usepackage[utf8]{inputenc}
\usepackage{amsmath,mathrsfs,amssymb,multirow,setspace,enumerate}
\usepackage[a4paper, total={7in, 9in}]{geometry}
\usepackage{xcolor}
\usepackage{lettrine}
\usepackage{tikz}
\usepackage[11pt]{moresize}
\theoremstyle{plain}
\usepackage{graphicx}
\usepackage{amsmath}
\newtheorem{theorem}{Theorem}[section]
\newtheorem{lemma}[theorem]{Lemma}
\newtheorem{proposition}[theorem]{Proposition}
\newtheorem{corollary}[theorem]{Corollary}

\theoremstyle{definition}

\theoremstyle{remark}
\newtheorem{remark}[theorem]{Remark}

\input xy
\xyoption{all}
\def\c{\Delta(SD_{8n})}

\begin{document}
\title[On The Commuting Graph of Semidihedral Group]{On The Commuting Graph of Semidihedral Group}
\author[Jitender Kumar,Sandeep Dalal, and Vedant Baghel]{ Jitender Kumar, Sandeep Dalal, and Vedant Baghel}
\address{Department of Mathematics, Birla Institute of Technology and Science Pilani, Pilani, India}
\email{jitenderarora09@gmail.com,deepdalal10@gmail.com, vedantbaghel4204@gmail.com}
\begin{abstract}
The commuting graph $\Delta(G)$ of a finite non-abelian group $G$ is a simple graph with vertex set $G$ and two distinct vertices $x, y$ are adjacent if $xy = yx$. In this paper, among some properties of $\Delta(G)$, we investigate $\c$ the commuting graph of the semidihedral group $SD_{8n}$. In this connection, we discuss various graph invariants of $\c$ including minimum degree, vertex connectivity, independence number, matching number and detour properties. We also obtain the Laplacian spectrum,  metric dimension and resolving polynomial of $\c$. 
\end{abstract}



\maketitle

 \section{Introduction}
The study of various algebraic structures through their graph-theoretic properties become an interesting research topic. The investigation of graphs related to various algebraic constructions
 is very important, because graphs of this type have valuable applications \cite{a.Chartrand-Resolvability,a.sebometricgenerators} and are related to automata theory  \cite{b.Kelerve-Automata,a.Kelerve-minimal-automata}. Let $G$ be a non-abelian group and $\Omega \subseteq G$. The \emph{commuting graph} of $G$, denoted by $\Delta(G, \Omega)$, is a simple graph in which the vertex set is $\Omega$ and two distinct vertices $x, y \in \Omega$ are adjacent whenever $xy = yx$. This graph is precisely the complement of the non-commuting graph of a group considered in \cite{2006-Abdollahi-non-commuting}. When $G = \Omega$, we denote it by $\Delta(G)$. Many authors have been studied commuting graphs for different choices of $G$ and $\Omega$ (see \cite{a.bundy2006connectivity,a.Rajat2016spectrum,a.mirzargar2014remarks}). The origin of this notion lies in a seminal paper by R. Brauer and K.A. Fowler \cite{a.1955-Brauer-group} who  were concerned primarily with the classification of finite simple groups when $G$ has even order and $ \Omega = G \setminus \{e\}$. Moreover, \cite{a.Segev1999homomorphic,a.segev2001commuting,a.segev2002anisotropic} use combinatorial parameters of certain
commuting graphs to establish long standing conjectures in the theory of division algebras.  The study of graph theoretic properties is now an exciting part of graph theory viz. spectrum theory, metric dimension etc. Motivated by the problem of uniquely determining the location of an intruder in a network, Slater \cite{a.Slater} introduced the concept of metric dimension which was then independently studied by Harary and Melter \cite{a.Hararymetricdimension}. It is a parameter that has appeared  in various applications of graph theory, as diverse as pharmaceutical chemistry \cite{b.Cameron-Design,a.Chartrand-Resolvability}, robot navigation \cite{b.khuller} etc.  While spectra of graphs have many applications in quantum chemistry, the spectral radius also plays a significant role in modelling of virus
propagation in computer networks. In particular, the eigenvalues of the adjacency matrix (notably, the spectral radius) are helpful tools to protect personal data in some databases.

In recent years, the commuting graphs of various algebraic structures have become a topic
of research for many mathematicians (see \cite{a.Ali-2016,a.Araujo2015symmetricinverse,a.Dolvazn2017,a.Shitov2018}). In 2011, Ara$\acute{\text{u}}$jo \emph{et al.} \cite{a.Araujo2011} calculated the diameter of commuting  graphs of various ideals of full transformation semigroup. Also, for every natural number $n \geq 2$, a finite semigroup whose commuting graph has diameter $n$ has been constructed in \cite{a.Araujo2011}. 
In \cite{a.Iramanesh2008S_n}, it was conjectured that the commuting graph of a finite group is either disconnected or has diameter bounded above by a constant independent of the group $G$. In 2019, Julio \emph{et al.} showed that the commuting graph of a finite group $G$ is connected strongly regular graph if and only if $G$ is isoclinic to an extraspecial $2$-group of order at least $32$. Also, they characterize the finite non-ableian groups $G$ for which the commuting graph of $G$ is disconnected strongly regular. Behnaz Tolue \cite{a.Behanz2019} introduced the twin non-commuting graph by partitioning the vertices of non-commuting graph and studied the graph theoretic properties of twin non-commuting graph of AC-group and dihedral group. The distant properties as well as detour distant properties of the commuting graph on the dihedral group $D_{2n}$ were investigated by Faisal \emph{et al.} in  \cite{a.Ali-2016}. Moreover, they obtained metric dimension of the commuting graph on $D_{2n}$ and its resolving polynomial. Ali \emph{et al.} \cite{a.Ali2019-spectra} studied the connectivity and the spectral radius of the commuting graph of dihedral and dicyclic groups. Recently, Vipul \emph{et al.} \cite{a.2018commuting-generalized-dihedral} studied the detour distance properties and obtained the resolving polynomial of the commuting graph of generalized dihedral group.  Motivated with the work of \cite{a.Ali2019-spectra,a.Ali-2016, a.2018commuting-generalized-dihedral}, in this paper we consider the commuting graphs in the context of semidihedral group $SD_{8n}$. 

This paper is structured as follows.  In section $2$, we provide necessary background material and fix our
notations used throughout the paper. In Section $3$, we study some properties of $\Delta(G)$. Section $4$ comprises the study of various graph invariants of $\c$ viz. Hamiltonian, perfectness, independence number, clique number, vertex connectivity, edge connectivity, vertex covering number, edge covering number etc.. Moreover, we study the Laplacian spectrum, metric dimension, resolving polynomial and the detour properties  of the commuting graph of $SD_{8n}$.

\section{Preliminaries}
In this section, we recall  necessary definitions, results and notations of graph theory from \cite{b.West}.
A graph $\Gamma$ is a pair  $ \Gamma = (V, E)$, where $V = V(\Gamma)$ and $E = E(\Gamma)$ are the set of vertices and edges of $\Gamma$, respectively. We say that two different vertices $a, b$ are $\mathit{adjacent}$, denoted by $a \sim b$, if there is an edge between $a$ and $b$.  The \emph{neighbourhood} $ N(x) $ of a vertex $x$ is the set all vertices adjacent to $x$ in $ \Gamma $. Additionally, we denote $N[x] = N(x) \cup \{x\}$. It is clear that we are considering simple graphs, i.e. undirected graphs with no loops  or repeated edges. If $a$ and $b$ are not adjacent, then we write $a \nsim b$. A subgraph of a graph $\Gamma$ is a graph $\Gamma'$ such that $V(\Gamma') \subseteq V(\Gamma)$ and $E(\Gamma') \subseteq E(\Gamma)$. If $U \subseteq V(\Gamma)$, then the  subgraph of $\Gamma$ induced by  $U$ is the graph $\Gamma'$ with vertex set $U$, and with two vertices adjacent in $\Gamma'$ if and only if they are adjacent in $\Gamma$.
A \emph{walk} $\lambda$ in $\Gamma$ from the vertex $u$ to the vertex $w$ is a sequence of  vertices $u = v_1, v_2,\cdots, v_{m} = w$ $(m > 1)$ such that $v_i \sim v_{i + 1}$ for every $i \in \{1, 2, \ldots, m-1\}$. If no edge is repeated in $\lambda$, then it is called a \emph{trail} in $\Gamma$. A trail whose initial and end vertices are identical is called a \emph{closed trail}. A walk is said to be a \emph{path} if no vertex is repeated. The length of a path is the number of edges it contains. A graph  $\Gamma$ is said to be \emph{connected} if there is a path between every pair of vertices. A graph $\Gamma$ is said to be \emph{complete} if any two distinct vertices are adjacent.  A path that begins and ends on the same vertex is called a \emph{cycle}. A cycle $C$ in a graph $\Gamma$ that includes every vertex of $\Gamma$ is called a \emph{Hamiltonian cycle} of $\Gamma$. If $\Gamma$ contains a Hamiltonian cycle, then $\Gamma$ is called a \emph{Hamiltonian graph}. The \emph{degree} of a vertex $v$ is the number of edges incident to $v$ and it is denoted as deg$(v)$. The smallest degree among the vertices of $\Gamma$ is called the \emph{minimum degree} of $\Gamma$ and it is denoted by $\delta(\Gamma)$. The \emph{chromatic number} $\chi(\Gamma)$ of a graph $\Gamma$ is the smallest positive integer $k$ such that the vertices of $\Gamma$ can be colored in $k$ colors so that no two adjacent vertices share the same color. A graph $\Gamma$ is \emph{Eulerian} if $\Gamma$ is both connected and has a closed trail (walk with no repeated edge) containing all the edges of a graph.
Suppose $\Gamma_1 = (V_1, E_1)$ and  $\Gamma_2 = (V_2, E_2)$ are graphs with disjoint vertex sets. 
The \emph{union} $\Gamma_1 \cup \Gamma_2$ is the graph with $V(\Gamma_1 \cup \Gamma_2) = V_1 \cup V_2$ and $E(\Gamma_1 \cup \Gamma_2) = E_1 \cup E_2$. For a positive integer $n$, we write $n \Gamma$ to denote the union of $n$ disjoint copies of $\Gamma$. The \emph{join} $\Gamma_1 \vee \Gamma_2$ is the graph with $V(\Gamma_1 \cup \Gamma_2) = V_1 \cup V_2$ and $E(\Gamma_1 \cup \Gamma_2) = E_1 \cup E_2 \cup \{(a, b) : a \in V_1, b \in V_2 \}$. 



A \emph{clique} of a graph  $\Gamma$ is a complete  subgraph  of $\Gamma$ and the number of vertices in a  clique of maximum size is called the \emph{clique number} of $\Gamma$ and it is denoted by $\omega({\Gamma})$. The graph $\Gamma$ is \emph{perfect} if $\omega(\Gamma') = \chi(\Gamma')$ for every induced subgraph $\Gamma'$ of $\Gamma$.
Recall that the {\em complement} $\overline{\Gamma}$ of $\Gamma$ is a graph with same vertex set as $\Gamma$ and distinct vertices $u, v$ are adjacent in $\overline{\Gamma}$ if they are not adjacent in $\Gamma$. A subgraph $\Gamma'$ of $\Gamma$ is called \emph{hole} if $\Gamma'$ is a  cycle as an induced subgraph, and $\Gamma'$ is called an \emph{antihole} of   $\Gamma$ if   $\overline{\Gamma'}$ is a hole in $\overline{\Gamma}$.

\begin{theorem}[\cite{a.strongperfecttheorem}]\label{strongperfecttheorem}
A finite graph $\Gamma$ is perfect if and only if it does not contain hole or antihole of odd length at least $5$.
\end{theorem}

\begin{remark}\label{rem.e}
Let $G$ be a finite group and $x$ is a dominating vertex. Then $x$ does not belong to the vertex set of any hole of length greater than $3$, or any antihole of $ \Delta(G)$.	
\end{remark}


An independent set of a graph $\Gamma$ is a subset  of $V(\Gamma)$ such that no two vertices in the subset are adjacent in $\Gamma$. The \emph{independence number} of  $\Gamma$ is the maximum size of an independent  set, it is denoted by $\alpha(\Gamma)$. 

A \emph{vertex  cut set} in a connected graph $\Gamma$ is a set of vertices  whose deletion increases the number of connected components of $\Gamma$. The \emph{vertex connectivity}  of a connected graph $\Gamma$ is the minimum size of a vertex  cut set and it is denoted by $\kappa(\Gamma)$ . For $k \ge 1$, graph $\Gamma$ is \emph{$k$-connected} if $\kappa(\Gamma) \ge k$. The \emph{edge cut set} and \emph{edge connectivity} can be defined analogously. It will be denoted as $\kappa'(\Gamma)$. It is well known that $\kappa(\Gamma) \le \kappa'(\Gamma) \le \delta(\Gamma)$. An \emph{edge cover} of a graph $\Gamma$ is a set $L$ of edges such that every vertex of $\Gamma$ is incident to some edge of $L$. The minimum cardinality of an edge cover in $\Gamma$ is called the \emph{edge covering number}, it is denoted by $\beta'(\Gamma)$. A \emph{vertex cover} of a graph $\Gamma$ is a set $Q$ of vertices such that it contains at least one endpoint of every edge of $\Gamma$. The minimum cardinality of a vertex cover in $\Gamma$ is called the \emph{vertex covering number}, it is denoted by $\beta(\Gamma)$. A \emph{matching} in a graph $\Gamma$ is a set of edges with no share endpoints and the maximum cardinality of a matching is called the \emph{matching number} and it is denoted by $\alpha'(\Gamma)$. We have the following equalities involving the above parameters.
\begin{lemma}\label{vertex-covering-matching number}
	Consider a graph $\Gamma$.
	\begin{enumerate}
		\item[(i)] $\alpha(\Gamma) + \beta(\Gamma) = |V(\Gamma)|$.
		\item[(ii)] If $\Gamma$ has no isolated vertices,  $\alpha'(\Gamma) + \beta'(\Gamma) = |V(\Gamma)|$.
	\end{enumerate}
\end{lemma}

The (\emph{detour}) \emph{distance}, $(d_D(u,v))$ $d(u, v)$, between two vertices $u$ and $v$ in a graph $\Gamma$ is the length of (longest) shortest $u-v$ path in $\Gamma$. The (\emph{detour}) \emph{eccentricity} of a vertex $u$, denoted by $(ecc_D(u))$ $ecc(u)$, is the maximum (detour) distance between $u$ and any vertex of $\Gamma$. The minimum (detour) eccentricity among the vertices of $\Gamma$ is called the (\emph{detour}) \emph{radius} of $\Gamma$, it is denoted by $(rad_D(\Gamma)$ $rad(\Gamma)$. The \emph{detour diameter} of a graph $\Gamma$ is the maximum detour eccentricity in $\Gamma$, denoted by $diam_D(G)$. A vertex $v$ is said to be \emph{eccentric vertex} for $u$ if $d(u, v) = ecc(u)$. A vertex $v$ is said to be an eccentric vertex of the graph $\Gamma$ if $v$ is an eccentric vertex for some vertex of $\Gamma$. A graph $\Gamma$ is said to be an \emph{eccentric graph} if every vertex of $\Gamma$ is an eccentric vertex. The \emph{centre} of $\Gamma$ is a subgraph of $\Gamma$ induced by the vertices having minimum eccentricity and it is denoted by $Cen(\Gamma)$. The \emph{closure} of a graph $\Gamma$ of order $n$ is the graph obtained from $\Gamma$ by recursively joining pairs of non-adjacent vertices whose sum of degree is at least $n$ until no such pair remains and it is denoted by $Cl(\Gamma)$.  The graph $\Gamma$ is said to be \emph{closed} if $\Gamma = Cl(\Gamma)$ (\cite{b.chartrand2004introduction}).

A vertex $v$ in a graph $\Gamma$ is a \emph{boundary vertex} of a vertex $u$ if $d(u, w) \leq d(u, v)$ for $w \in$N$(v)$, while a vertex $v$ is a boundary vertex of a graph $\Gamma$ if $v$ is a boundary vertex of some vertex of  $\Gamma$. The subgraph $\Gamma$ induced by its boundary vertices is the \emph{boundary} $\partial(\Gamma)$ of $\Gamma$. A vertex $v$ is said to be a \emph{complete vertex} if the subgraph induced by the neighbors of $v$ is complete. A vertex $v$ is said to be an \emph{interior vertex} of a graph $\Gamma$ if for each $u \ne v$, there exists a vertex $w$  and a path $u-w$ such that $v$ lies in that path at the same distance from both $u$ and $w$. A subgraph induced by the interior vertices of $\Gamma$ is called \emph{interior} of $\Gamma$ and it is denoted by $Int(\Gamma)$. 

\begin{theorem}[\cite{b.chartrand2004introduction}, p.337]\label{SD_8n-boundary-complete}
Let $\Gamma$ be a connected graph and $v \in V(\Gamma)$. Then $v$ is a complete vertex of $\Gamma$  if and only if $v$ is a  boundary vertex of $x$ for all $x \in V(\Gamma) \setminus \{v\}$.
\end{theorem}

\begin{theorem}[\cite{b.chartrand2004introduction}, p.339]\label{SD_8n-boundary-interior}
Let $\Gamma$ be a connected graph and $v \in V(\Gamma)$. Then $v$ is a  boundary vertex of $\Gamma$  if and only if $v$ is not an interior vertex of $\Gamma$.
\end{theorem}

For  vertices $u$ and $v$ in a graph $\Gamma$, we say that $z$ \emph{strongly resolves} $u$ and $v$ if there exists a shortest path from $z$ to $u$ containing $v$, or a shortest path from $z$ to $v$ containing $u$. A subset $U$ of $V(\Gamma)$ is a \emph{strong resolving set} of $\Gamma$ if every pair of vertices of $\Gamma$ is strongly resolved by some vertex of $U$. The least cardinality of a strong resolving set of $\Gamma$ is called the \emph{strong metric dimension} of $\Gamma$ and is denoted by $\operatorname{sdim}(\Gamma)$. For  vertices $u$ and $v$ in a graph $\Gamma$, we write $u\equiv v$ if $N[u] = N[v]$. Notice that that $\equiv$ is an equivalence relation on $V(\Gamma)$.
We denote by $\widehat{v}$ the $\equiv$-class containing a vertex $v$ of $\Gamma$.
Consider a graph $\widehat{\Gamma}$ whose vertex set is the set of all $\equiv$-classes, and vertices $\widehat{u}$ and  $\widehat{v}$ are adjacent if $u$ and $v$ are adjacent in $\Gamma$. This graph is well-defined because in $\Gamma$, $w \sim v$ for all $w \in \widehat{u}$ if and only if $u \sim v$.  We observe that $\widehat{\Gamma}$ is isomorphic to the subgraph $\mathcal{R}_{\Gamma}$ of $\Gamma$ induced by a set of vertices consisting of exactly one element from each $\equiv$-class. Subsequently, we have the following result of \cite{a.strongmetricdim2018} with $\omega(\mathcal{R}_{\Gamma})$ replaced by $\omega(\widehat{\Gamma})$.

\begin{theorem}[{\cite[Theorem 2.2 ]{a.strongmetricdim2018}}]\label{strong-metric-dim}
Let $\Gamma$ be a graph with diameter $2$. Then  sdim$(\Gamma) = |V(\Gamma)| - \omega(\widehat{\Gamma})$.
\end{theorem}

When $\Gamma = \c$ for some group $G$, we denote $\widehat{\Gamma}$ by $\widehat{\Delta}(G)$.

\section{Commuting graph of a finite group}
In this section,  we investigate the commuting graph of an arbitrary group $G$. First, we show that the edge connectivity and the minimum degree of $\Delta(G)$ are equal. For $a \in G$, let $cl(a)$ be the conjugacy class of $G$ containing $a$. The centre of $G$ is denoted by $Z(G)$ and the centralizer of the element $a$ is denoted by  $C_G(a) = \{b \in G: ab = ba \}$. The following remark follows from the definition of $\Delta(G)$.

\begin{remark}\label{Nbd-central-element}
In the commuting graph of a group $G$, we have {\rm N}$[x] = C_G(x)$ for each $x \in G$. 
\end{remark}

\begin{theorem}\label{edge-connectivity}
Let $G$ be a finite group and $t = $ {\rm max}$\{|cl(a)| : a \in G\}$. Then \[\kappa'(\Delta(G)) = \delta(\Delta(G)) = \frac{|G|}{t} - 1.\]
\end{theorem}

\begin{proof}
In view of Remark \ref{Nbd-central-element}, $\delta(\Delta(G)) = r - 1$, where $r = $ {\rm min}$\{|C_G(a)| : a \in G\}$. For a graph $\Gamma$, since $\kappa'(\Gamma) \le \delta(\Gamma)$ we obtain $\kappa'(\Delta(G)) \leq r - 1$. By Menger's theorem  (cf. \cite[Theorem 3.2]{b.bondy1976graph}), to prove another inequality, it is sufficient to show that there exist at least $r - 1$ internally edge disjoint paths between arbitrary pair of vertices. Let $x$ and $y$ be the distinct pair of vertices in $\Delta(G)$. Suppose $|C_G(x) \cap C_G(y)|  = q$. For $z \in C_G(x) \cap C_G(y)$, we have $x \sim z$ and $y \sim z$.  Then $\Delta(G)$ contains at least $q$ internally edge disjoint paths between $x$ and $y$.  Further there exist $x_1, x_2, \ldots, x_{r-q -1} \in C_G(x) \setminus C_G(y)$ and  $y_1, y_2, \ldots, y_{r-q - 1} \in C_G(y) \setminus C_G(x)$. Consequently, we get $x \sim x_i \sim e \sim y_i \sim y$  internally edge disjoint paths between $x$ and $y$ which are $r-q - 1$ in total. Thus, we have at least $r - 1$ internally disjoint paths between $x$ and $y$. Since for $x \in G$, we have $|cl(x)| = \frac{|G|}{|C_G(x)|}$. Hence,  $\kappa'(\Delta(G)) = \delta(\Delta(G)) = t - 1$,  where $t = $ {\rm max}$\{|cl(a)| : a \in G\}$.
\end{proof}

In the following theorem, we obtain the matching number of $\Delta(G)$.

\begin{theorem}\label{matching}
Let $G$ be a finite group and let $t$ be the number of involutions in $G \setminus Z(G)$. Then 
\begin{enumerate}[\rm (i)]
\item $\alpha'(\Delta(G)) =  \left\{ \begin{array}{ll}
\dfrac{|G|-1}{2} & \text{if $|G|$ is odd};\\
\vspace{0.1cm}\\
\dfrac{|G|}{2}  & \text{if  $|G|$ is even and $t \leq |Z(G)|$;}\end{array}\right.$
\vspace{0.1cm}
\item if $G$ is of even order and $t > |Z(G)|$, then $$\dfrac{|G| + |Z(G)| - t}{2} \leq \alpha'(\Delta(G)) \leq \dfrac{|G|}{2}.$$
\end{enumerate}

\end{theorem}

\begin{proof}
(i) Let $G$ be a finite group of odd order. Observe that for  $x\in G {\setminus} \{e\}$, we have $x \ne x^{-1}$ as $o(x) > 2$ and $x \sim x^{-1}$. Thus, $M = \{(x, x^{-1}) : x \ne e \in G\}$ is a matching of order $\dfrac{|G|-1}{2}$ in $\Delta(G) $. On the other hand, the order of a largest matching in a graph of order $n$ is $\left \lfloor \dfrac{n}{2} \right \rfloor$. Hence we get $\alpha'(\Delta(G)) = \dfrac{|G|-1}{2}$.

Now we assume that $G$ is of even order and $t \leq |Z(G)|$. Note that $x \in Z(G)$ if and only if $x^{-1} \in Z(G)$. Consider the set $A = \{a \in G \setminus Z(G) : o(a) = 2\}$. Suppose $t \leq |Z(G)|$. Then we denote the edges with ends $a_i$ and $z_i$ by $\epsilon_i$, where $a_i \in A$ and $z_i \in Z(G)$.  Let $M = \{\epsilon_i : 1 \leq i \leq t \} \bigcup \{(x, x^{-1}) : x \ne e \in G \setminus Z(G)\}$ is a matching such that $G \setminus G_M \subseteq Z(G)$ and $|G \setminus G_M|$ is even, where $G_M = \{ x \in G : (\exists \; x' \in G), \; (x, x') \in M \}$. For $|G|$ is even, we get $|G_M|$ is even so $\mathcal M = M \cup \{(x, x') : x \ne x' \; \text{and} \; x, x' \in G_M \}$ is a matching of order $\dfrac{|G|}{2}$.  Since $\alpha'(\Delta(G)) \leq \dfrac{|G|}{2}$, we have the result.

\smallskip

\noindent
(ii) Suppose $|G|$ is even and $t > |Z(G)|$. By the proof of part (i), we have a matching $\mathcal M$ of size at least $\dfrac{|G| + |Z(G)| - t}{2}$. Thus, we get the desired inequality.
\end{proof}

In view of Lemma \ref{vertex-covering-matching number}(ii), we have the following corollary. 

\begin{corollary}\label{edge.cover}
For a finite group $G$ and let $t$ be the number of involutions in $G \setminus Z(G)$, we have
\begin{enumerate}[\rm (i)]
\item $\beta'(\Delta(G)) =  \left\{ \begin{array}{ll}
\dfrac{|G|+1}{2} & \text{if $|G|$ is odd};\\
\vspace{0.1cm}\\
\dfrac{|G|}{2}  & \text{if  $|G|$ is even and $t \leq |Z(G)|$;}\end{array}\right.$
\vspace{0.1cm}
\item if $G$ is of even order and $t > |Z(G)|$, then $$ \dfrac{|G|}{2} \leq \beta'(\Delta(G)) \leq \dfrac{|G| - |Z(G)| + t}{2} .$$
\end{enumerate}
\end{corollary}

For $x \in G \setminus Z(G)$, $C_G(x)$ is called \emph{maximal centralizer} if there is no $y \in G \setminus Z(G)$ such that $C_G(x)$ is a proper subgroup of $C_G(y)$. In the following theorem, we compute the vertex connectivity of $\Delta(G)$.

\begin{lemma}\label{vertex-connectivity}
Let $G$ be a finite non-abelian group such that, for some $x \in G$, $C_G(x)$ is a maximal centralizer and an abelian subgroup of $G$. Then $\kappa(\Delta(G)) = |Z(G)|$.
\end{lemma}

\begin{proof}
Suppose $x \in G$ such that $C_G(x)$ is an abelian subgroup of $G$. For $y \in C_G(x)$, we have $xy = yx$. Since $C_G(x)$ is an abelian subgroup of $G$ so that $C_G(x) \subseteq C_G(y)$. Consequently, $C_G(x) = C_G(y)$ for all $y \in C_G(x)$ as $C_G(x)$ is a maximal centralizer. For $G$ is a non-abelian group, there exists $z \in G$ such that $xz \ne zx$. It follows that there is no path between $x$ and $z$ in the subgraph induced by the vertices of $G \setminus Z(G)$. For instance if there is a path $x = x_1 \sim x_2 \sim \ldots \sim x_r = y$ for some $r > 1$ gives $x_i \in C_G(x)$ for all $i$ where, $1 \leq i \leq r$. As a result, $xz = zx$; a contradiction. Thus, the subgraph induced by the vertices of $G \setminus Z(G)$ is disconnected so $\kappa(\Delta(G)) \leq |Z(G)|$.

On the other hand, observe that any vertex cut-set contains $Z(G)$. Otherwise, there exists a vertex cut set $\mathcal O$ does not contain $Z(G)$. Then there is $a \in Z(G)$  such that $a \notin \mathcal O$. For distinct $x, y \in G \setminus \mathcal O$, we have $x \sim a \sim y$. Consequently, the subgroup induced by the vertices of $G \setminus \mathcal O$ is connected implies the set $\mathcal O$ is not a vertex cut-set; a contradiction. Thus, $\kappa(\Delta(G)) \geq |Z(G)|$ and hence  $\kappa(\Delta(G)) = |Z(G)|$.
\end{proof}

 A group $G$ is called an $AC$-group if the centralizer of every non-central element is abelian.

\begin{theorem}
Let $G$ be a finite group such $\Delta(G) \cong \Delta(H)$ for some $AC$-group $H$. Then $G$ is an $AC$-group.
\end{theorem}

\begin{proof}
Suppose $\Delta(G) \cong \Delta(H)$ where $H$ is an $AC$-group. By \cite[Lemma 2.1]{a.Rajat2016spectrum}, the subgraph induced by the vertices of $H \setminus Z(H)$ is $\bigcup\limits_{ i=1}^{r} K_{|X_i| - |Z(H)|}$, where $X_1, X_2, \ldots, X_r$ are the distinct centralizers of non-central elements of $H$. Note that for each $x \in Z(H)$, we get $x \sim y$ for all $y \ne x \in H$. Therefore, we have $\Delta(G) \cong \Delta(H) = K_{|Z(H)|} \vee \bigcup\limits_{ i=1}^{r} K_{|X_i| - |Z(H)|}$. If $x \in G \setminus Z(G)$, then clearly N$[x] = X_i$ for some $i$ gives  $C_G(x)$ is an abelian subgroup of $G$. Thus, $G$ is an $AC$-group.
\end{proof}

\begin{proposition}\label{clique}
Let $K$ be a clique in $\Delta(G)$. Then $\omega(\Delta(G)) = |K|$ if and only if $K$ is a commutative subgroup of maximum size in $G$.
\end{proposition}
\begin{proof}
Let $K$ be any clique of maximum size such that $x, y \in K$. Then $xy$ commutes to every element of $K$. Consequently, we get $xy \in K$ as $|K|$ is maximum. Note that $e \in K$, where $e$ is the identity element of $G$. If $x \in K$, then $x^{-1} \in \langle x \rangle$ so $x^{-1}$ commutes with every element of $K$. Since $K$ is a clique of maximum size we obtain $x^{-1} \in K$. Therefore $K$ forms a subgroup of $G$. Clearly, $K$ is a commutative subgroup of maximum size. Converse part is straightforward. \end{proof}

The proof of the following lemma follows from the  definition of complete vertex.
\begin{lemma}\label{complete-vertex}
Let $x \in G$. Then $x$ is a complete vertex in $\Delta(G)$ if and only if $C_G(x)$ is a commutative subgroup of $G$. 
\end{lemma}

\begin{proposition}\label{eccentricity}
For any group $G$, we have 
\[Ecc(\Delta(G)) =  \left\{ \begin{array}{ll}
\Delta(G) \setminus \{e\} & \text{if $|Z(G)| = 1$};\\
\Delta(G)  & \text{\rm Otherwise.}\end{array}\right.\]
\end{proposition}

\begin{proof}
Let $x \in G \setminus \{e\}$. Then $d(x, e) = ecc(e)$ so that $x$ is an eccentric vertex for $e$. Thus, each non identity element of $G$ is an eccentric vertex of $\Delta(G)$. If $|Z(G)| > 1$, then there exists $x \in Z(G) \setminus \{e\}$. Note that $e$ is an eccentric vertex for $x$. Thus the result holds. For $Z(G) = \{e\}$, note that $e$ is not an eccentric vertex of $\Delta(G)$. For instance, if $e$ is an eccentric vertex for some $y \in G$, then $d(x, e) =  1 = ecc(y)$. As $y \in G \setminus Z(G)$, there exists $z \in G$ such that $z \nsim y$ gives $ecc(y) > 1$; a contradiction.
\end{proof}


\begin{corollary}\label{eccenteric-center-gretaer-than-1}
Let $G$ be a group with $|Z(G)| > 1$. Then $\Delta(G)$ is an eccentric graph.
\end{corollary}

In the next lemma, for each $x \in G$, we obtain the condition on $y \in G$ such that $y$ is a boundary vertex of $x$.
 
\begin{lemma}\label{boundary-vertex}
Let $x \in G$. Then $y$ is a boundary vertex of $x$ in $\Delta(G)$ if and only if one of the following  holds:
\begin{enumerate}[\rm (i)]
\item $y \notin C_G(x)$.
\item $C_G(y) \subseteq C_G(x)$.
\end{enumerate}
\end{lemma}

\begin{proof} If $x \in Z(G)$, then by the definition of boundary vertex it is routine to verify that the result holds. Now, let $x \in G \setminus Z(G)$. Suppose $y$ is a boundary vertex of $x$. On contrary, we assume that $y \in C_G(x)$ and $C_G(y) \nsubseteq C_G(x)$. Then there exists $z \in C_G(y)$ such that $z \notin C_G(x)$. Consequently, we get $d(x, z) > 1$ and $d(x, y) = 1$; a contradiction. On the other hand we assume that $y$ satisfy either (i) or (ii). Suppose $y \notin C_G(x)$. Since $e$ is adjacent to all the vertices of $\Delta(G)$ so diameter of $\Delta(G)$ is at most two. Therefore, $d(x, y) = 2$ implies $d(x, y) \geq d(x, z)$ for all $z \in G$. Consequently, $y$ is a boundary vertex of $x$. If $C_G(y) \subseteq C_G(x)$, then clearly $y$ is a boundary  vertex of $x$.
\end{proof}

\begin{proposition} For the graph $\Delta(G)$, we have 
$\partial(\Delta(G)) = Ecc(\Delta(G)).$
\end{proposition}

\begin{proof}
For $x \ne e$, we have $C_G(x) \subseteq C_G(e)$, so by Lemma \ref{boundary-vertex}, $x$ is a boundary vertex of $e$. Therefore, $x$ is a boundary vertex of $\Delta(G)$. If $|Z(G)| > 1$, there exists $x \ne e \in Z(G)$ so $e$ is a boundary vertex of $\Delta(G)$. For $Z(G) = \{e\}$, note that $e$ is not a boundary vertex of $\Delta(G)$. For instance, if $e$ is a boundary vertex of $x$ for some $x \in G \setminus \{e\}$, then d$(x, y) \leq {\rm d}(x, e) =1$ for all $y \in  {\rm N}(e) = G \setminus \{e\}$. Consequently, we get $x \in Z(G)$; a contradiction. Thus, the result holds.
\end{proof}

\begin{lemma}[{\cite[Lemma 1.2 ]{a.Ali-2016}}]\label{center}
For any group $G$, $Cen(\Delta(G)) = Z(G)$.\end{lemma}

 Further we characterise the group $G$ such that $Int(\Delta(G)) = Cen(\Delta(G))$.

\begin{proposition}\label{center-Interior}
Let $G$ be a non-abelian group with $|G| > 2$. Then we have $Int(\Delta(G)) = Cen(\Delta(G))$ if and only if $G$ satisfies the following condition.
\begin{enumerate}[\rm (i)]
\item $G$ is an $AC$-group.
\item $|C_G(x)| > 2$ for all $x \in G$.
\end{enumerate}
\end{proposition}

\begin{proof}
In view of Lemma \ref{center}, we show that $Int(\Delta(G)) = Z(G)$ if and only $G$ satisfy (i) and (ii). First we assume that $G$ holds (i) and (ii). We claim that $v$ is an interior point if and only if $v \in Z(G)$. Let $v \in G \setminus Z(G)$. Since $C_G(v)$ is a commutative subgroup of $G$ so by Theorems \ref{SD_8n-boundary-complete}, \ref{SD_8n-boundary-interior} and Lemma \ref{complete-vertex}, $v$ is not an interior point of $\Delta(G)$. On the other hand, we assume that $v \in Z(G)$. Then clearly $v$ is not a complete vertex as $G$ is a non-abelian group. In view of Theorems \ref{SD_8n-boundary-complete} and \ref{SD_8n-boundary-interior}, $v$ is an interior point. Thus,  $Int(\Delta(G)) = Z(G)$.

Suppose $Int(\Delta(G)) = Z(G)$. Let $x \in G \setminus Z(G)$. Then $x$ is not an interior point as $Int(\Delta(G)) = Z(G)$ implies $x$ is a complete vertex (cf. Theorems \ref{SD_8n-boundary-complete}, \ref{SD_8n-boundary-interior} and Lemma \ref{complete-vertex}). Consequently, $C_G(x)$ is an abelian subgroup of $G$. Thus, $G$ is an $AC$-group. Now we show that $|C_G(x)| > 2$ for all $x \in G$. If $|C_G(x)| = 2$ for some $x \ne e \in G$, then there exists $y \in G$ such that $x \nsim y$ and $x, y \in$ N$(e)$. As a result, $e$ is not a complete vertex of $(\Delta(G)$. By Theorems \ref{SD_8n-boundary-complete} and \ref{SD_8n-boundary-interior}, $e$ is not an interior  point of $\Delta(G)$; a contradiction.  Now we assume that $C_G(x)$ is not a commutative subgroup of $G$ for some $x \in G \setminus Z(G)$. Then $x$ is not a complete vertex of $\Delta(G)$ gives $x$ is an interior vertex of $\Delta(G)$; a contradiction of the fact  $Int(\Delta(G)) = Z(G)$. Thus, the result holds.
\end{proof}

\section{Commuting graph of the semidihedral group $SD_{8n}$}
In this section we obtain various graph invaraints of $\c$ viz. vertex connectivity, independence number, edge connectivity,  matching number, clique number etc. As a consequence, we obtain the vertex covering number and the edge covering number of $\c$. Further, we also study the Laplacian spectrum, resolving polynomial and the detour properties of $\c$ in various subsections. For $n \geq 2$, the \emph{semidihedral group} $SD_{8n}$ is a group of order $8n$ with presentation  $$SD_{8n} = \langle a, b  :  a^{4n} = b^2 = e,  ba = a^{2n -1}b \rangle.$$
First note that $$ ba^i = \left\{ \begin{array}{ll}
a^{4n -i}b & \mbox{if $i$ is even,}\\
a^{2n - i}b& \mbox{if $i$ is odd.}\end{array} \right.$$
Thus, every element of $SD_{8n} {\setminus} \langle a \rangle$ is of the form $a^ib$ for some $0 \leq i \leq 4n-1$. We denote the subgroups $H_i = \langle a^{2i}b \rangle = \{e, a^{2i}b\}$ and $ T_j =  \langle a^{2j + 1}b \rangle = \{e, a^{2n}, a^{2j +1}b, a^{2n + 2j +1}b\} $. Then we have $$SD_{8n} = \langle a \rangle \cup \left( \bigcup\limits_{ i=0}^{2n-1} H_i \right) \cup \left( \bigcup\limits_{ j=  0}^{n-1} T_{j}\right).$$

Further, $Z(SD_{8n}) = \left\{ \begin{array}{ll}
\{e, a^{2n}\} & \text{when $n$ is even,}\\
\{e, a^n, a^{2n}, a^{3n}\}& \text{otherwise.}\end{array} \right.$\\
By Remark \ref{Nbd-central-element}, we have the following lemma.

\begin{lemma}\label{dominating-vertex} In $\c$, 
\begin{enumerate}[\rm(i)]
\item  for even $n$, we have {\rm N}$[x]= SD_{8n}$ if and only if $x \in \{e, a^{2n}\}$.

\item for odd $n$, we have {\rm N}$[x]= SD_{8n}$ if and only if $x \in \{e, a^n, a^{2n}, a^{3n}\}$.
\end{enumerate}
\end{lemma}

By \cite[Theorem 1.2.26]{b.West}, we have the following corollary. 

\begin{corollary}
The commuting graph of $SD_{8n}$ is not Eulerian.
\end{corollary}

The following remarks will be useful in the sequel.

\begin{remark}\label{r.even-a^ib}
For even $n$ and $1 \leq i \leq 4n$, we have $a^ib$ commutes with $a^jb$ if and only if $j = 2n + i$.
\end{remark}

\begin{remark}\label{r.even-a^i}
For even $n$ and $1 \leq i \leq 4n$, we have  $a^ib$ commutes with $a^j$ if and only if $j \in \{ 2n, 4n\}$.
\end{remark}

\begin{remark}\label{r.odd-a^i}
For odd $n$ and $1 \leq i \leq 4n$, we have  $a^ib$ commutes with $a^j$ if and only if $j \in \{n, 2n, 3n, 4n\}$.
\end{remark}

\begin{remark}\label{r.odd-a^ib}
For odd $n$ and $1 \leq i \leq 4n$, we have $a^ib$ commutes with $a^jb$ if and only if\\ $j \in \{n + i, 2n + i, 3n +i\}$.
\end{remark}

In view of the  Remarks \ref{r.even-a^ib} -  \ref{r.odd-a^ib}, we obtain the neighbourhood of each vertex of $\c$.

\begin{lemma}\label{nbd-even}
In $\c$, for even $n$, we have
\begin{enumerate}[\rm(i)]
\item  {\rm N}$[x] = \{e, a^{2n}, a^ib, a^{2n + i}b\}$ if and only if $x \in \{a^i b, a^{2n +i}b\}$, where $1 \leq i \leq 4n$.
\item  {\rm N}$[x] = \langle a \rangle$ if and only if $x \in \langle a \rangle \setminus \{e, a^{2n}\}$.
\end{enumerate}
\end{lemma}

\begin{lemma}\label{nbd-odd}
In $\c$, for odd $n$, we have
\begin{enumerate}[\rm(i)]
\item {\rm N}$[x] = \{e, a^{n}, a^{2n}, a^{3n}, a^ib, a^{n + i}b,  a^{2n + i}b, a^{3n + i}b\}$ if and only if $x \in \{a^ib, a^{n + i}b,  a^{2n + i}b, a^{3n + i}b\}$, where $1 \leq i \leq 4n$.
\item {\rm N}$[x] = \langle a \rangle$ if and only if $x \in \langle a \rangle \setminus \{e, a^n, a^{2n}, a^{3n}\}$.
\end{enumerate}
\end{lemma}

In view of Section 3 and Lemmas \ref{nbd-even}, \ref{nbd-odd}, we have the following proposition.

\begin{proposition}
The commuting graph of $SD_{8n}$ satisfies the following properties:

\begin{enumerate}[\rm (i)]
\item  sdim$(\c)$ is $8n-2$.

\item  $\omega(\c) = 4n$.

\item $Cen(\c) = Int(\c)$.

\item  $\c$ is an eccentric graph.

\item  $\c$ is a closed graph.

\item $\kappa'(\c) = \left\{ \begin{array}{ll}
3 & \text{if $n$ is even };\\
7 & \text{if $n$ is odd.}\end{array}\right.$
\end{enumerate}
\end{proposition}

\begin{proof}(i) One can observe that the graph $\widehat{\Delta}(SD_{8n})$ is a star graph. Thus by Theorem \ref{strong-metric-dim}, the result holds.

\smallskip

\noindent
(ii) For $1 \leq i \leq 4n$, note that the element $a^ib$ is commute with at most eight elements of $\c$. Since the commutative subgroup generated by $a$ is of size $4n$. It follows that any commutative subgroup of $SD_{8n}$ of maximum size does not contain the elements of the form $a^ib$. Thus, $\langle a \rangle$ is a commutative subgroup of $SD_{8n}$ of maximum size $4n$ and hence the result holds (cf. Proposition \ref{clique}).

\smallskip

\noindent
(iii) For any $x \in SD_{8n} \setminus Z(SD_{8n})$, we have N$[x] = \langle x \rangle$. By Remark \ref{Nbd-central-element}, $C_G(x)$ is a commutative subgroup of $SD_{8n}$. Thus by Proposition \ref{center-Interior},   $Int(\c) = Cen(\c)$.

\smallskip

\noindent
(iv) Since $|Z(\c)| > 1$ (see Lemma \ref{dominating-vertex}) so that by Proposition \ref{eccentricity}, the result holds. 
\smallskip

\noindent
(v) Note that for non-adjacent vertices $x$ and $y$, we have $|\text{N}(x)|+|\text{N}(y)| <   |V(\c)|= 8n$ (cf. Lemmas \ref{nbd-even} and \ref{nbd-odd}). Consequently, deg$(x) + \text{deg}(y) < |V(\c)|$ for all non-adjacent vertices $x$ and $y$. Thus, by {\cite[Lemma 2.15 ]{a.Ali-2016}}, the result hold.

\smallskip

\noindent
(vi) For even $n$,  by Lemma \ref{nbd-even} note that $\delta(\c) = $ min$\{ |C_G(x)| : x \in SD_{8n} \} - 1 = 3$ and for odd $n$, by Lemma \ref{nbd-odd}, note that $\delta(\c) = $ min$\{ |C_G(x)| : x \in SD_{8n} \} - 1 = 7$. Thus, by Theorem \ref{edge-connectivity}, we have the result.  
\end{proof}

As a consequence of Lemmas \ref{nbd-even} and \ref{nbd-odd}, we have the following proposition.

\begin{proposition}\label{structure-SD_8n}
For $n \ge 1$, we have
\[ \c \cong \left\{ \begin{array}{ll}
K_2 \vee (K_{4n -2} \cup 2nK_2) & \mbox{if $n$ is even };\\
K_4 \vee (K_{4n -4} \cup nK_4) & \mbox{if $n$ is odd.}\end{array}\right.\]
\end{proposition}

Now, we obtain the automorphism group of the commuting graph of $SD_{8n}$. An \emph{automorphism} of a graph $\Gamma$ is a permutation $f$ on  $V(\Gamma)$ with the property that, for any vertices $u$ and $v$, we have $uf \sim vf$ if and only if $u \sim v$. The set ${\rm Aut}(\Gamma)$ of all graph automorphisms of a graph $\Gamma$ forms a group with respect to  composition of mappings. The symmetric group of degree $n$ is denoted by $S_n$. 

\begin{theorem}[{\cite[Theorem 2.2]{a.2017ashrafi_automorphism}}] \label{Aut-join-graphs}
Suppose $\Gamma = n_1\Gamma_1 \cup n_2\Gamma_2\cup\cdots \cup n_t \Gamma_t$ with $\Gamma_i \ne \Gamma_ j$ for $i \ne j$. Then {\rm Aut}$(\Gamma) = {\rm Aut}(\Gamma_1) \wr S_{n_1} \times {\rm Aut}(\Gamma_2) \wr S_{n_2} \times \cdots \times {\rm Aut}(\Gamma_t) \wr S_{n_t}$.
\end{theorem}

\begin{remark}
If $A$ is the set of all vertices adjacent to every vertex in a graph $\Gamma$ and $\Gamma- A$ is the subgraph of $\Gamma$ induced by the vertices of $V(\Gamma) -A$, then {\rm Aut}$(\Gamma)$ is isomorphic to $S_{|A|} \times {\rm Aut}(\Gamma - A)$.
\end{remark}

By Proposition \ref{structure-SD_8n} and Theorem \ref{Aut-join-graphs}, we have the following theorem.

\begin{theorem}
For $n \in \mathbb N$, we have
	\[{\rm Aut}(\c) = \left\{ \begin{array}{ll}
	S_2 \times \left( (S_{4n-2} \wr S_1) \times (S_2 \wr S_{2n}) \right) & \mbox{if $n$ is even };\\
	S_4 \times \left( (S_{4n-4} \wr S_1) \times (S_4 \wr S_{n}) \right) & \mbox{if $n$ is odd.}\end{array}\right.  \]
\end{theorem}

Next, we obtain the vertex connectivity, independence number and the matching number of $\c$.

\begin{theorem}\label{independence}
In the graph $\c$,
\end{theorem}
\begin{enumerate}
\item[(i)] the vertex connectivity of $\c$ is given below: \[\kappa(\c) = \left\{ \begin{array}{ll}
2 & \text{if $n$ is even };\\
4 & \text{if $n$ is odd.}\end{array}\right.\]
\\

\item[(ii)] the independence number of $\c$ is given below:
\[\alpha(\c) = \left\{ \begin{array}{ll}
1 & \text{if $n = 1$  };\\
2n + 1 & \text{if $n$ is even };\\
n + 1 & \text{otherwise.}\end{array}\right.\]	

\item[(iii)] the matching number of $\c$ is $4n$.
\end{enumerate}

\begin{proof}
(i) In view of Remark \ref{Nbd-central-element} and Lemmas \ref{nbd-even}, \ref{nbd-odd}, $C_G(ab)$ is an abelian subgroup of $SD_{8n}$ and maximal centralizer. By Theorem \ref{vertex-connectivity}, $\kappa(\c) = |Z(SD_{8n})|$. Thus, we have the result.
 \smallskip

\noindent
(ii) Suppose $n$ is even. Consider the set $I = \{a^ib : 1 \leq i \leq 2n \} \cup \{a\}$.  In view of Lemmas \ref{dominating-vertex} and \ref{nbd-even}, $I$ is an independent in $\c$ of size $2n +1$. If there exists another independent set $I'$ such that $|I'| > 2n + 1$, then there exist $x, y \in I'$ such that either $x, y \in \langle a \rangle$ or $x, y \in \{a^ib, a^{2n +i}b\}$ for some $i$, where $1 \leq i \leq 2n$ as $SD_{8n} = \langle a \rangle \cup (\bigcup\limits_{i = 1}^{2n} \{a^ib, a^{2n +i}b \})$. In both cases, we have $x \sim y$ (see Lemma \ref{nbd-even}); a contradiction of the fact that $I'$ is an independent in $\c$. Thus the result holds. 

On the other hand, we assume that $n$ is odd. By using the Lemma \ref{nbd-odd} and similar to even $n$ case, we get an independent set $I = \{a^ib : 1 \leq i \leq n \} \cup \{a\}$ of the maximum size $n +1$. 

 \smallskip

\noindent
(iii)  In view of Lemmas \ref{nbd-even} and \ref{nbd-odd}, $a^ib \sim a^{2n + i}b$ for all $i$, where $1 \leq i \leq 2n$. Consider the set $\mathcal M = \{(a^ib, a^{2n +i}b) \in E(\c) : 1 \leq i \leq 2n \} \cup \{(a^i, a^{2n +i}) \in E(\c): 1 \leq i \leq 2n \}$ which forms a matching of size $4n$. Consequently, we get $\alpha'(\c) \geq 4n$. It is well known that  $\alpha'(\c) \leq \frac{|V(\c)|}{2} = 4n$. Thus, $\alpha'(\c) = 4n$.
\end{proof}

In view of Lemma \ref{vertex-covering-matching number}, we have the following corollary.

\begin{corollary}
For $n \geq 1$, 
\begin{enumerate}[\rm (i)]
\item the vertex covering number of $\c$ is given below:
\[\beta(\c) = \left\{ \begin{array}{ll}
7 & \text{if $n = 1$};\\
6n - 1 & \text{if $n$ is even };\\
7n - 1 & \text{otherwise.}\end{array}\right.\]
		
\item the edge covering number of $\c$ is $4n$.
\end{enumerate}
	
\end{corollary}

Now, we investigate  perfectness and Hamiltonian property of $\c$.

\begin{theorem}
The commuting graph of $SD_{8n}$ is perfect.	
\end{theorem}

\begin{proof} In view of Theorem \ref{strongperfecttheorem}, it is enough to show that $\c$ does not contain a hole or antihole of odd length at least five. Note that neither hole nor antihole can contain any element of $Z(\c)$ (cf. Remark \ref{rem.e}). First suppose that $\c$ contains a hole $C$ given by $x_1 \sim x_2\sim \dots \sim  x_{2l + 1} \sim x_1$, where $l \geq 2$. Note that any hole can contain at most two elements of $\langle a \rangle$, otherwise $C$ contains a triangle which is not possible.  In view of Lemmas \ref{nbd-even} and \ref{nbd-odd}, N$[x] = \langle a \rangle$ if and only if $ x \in \langle a \rangle \setminus Z(\c)$. It follows that $x_i \notin \langle a \rangle$ for all $i$, where $1 \leq i \leq 2l + 1$. Consequently, we get  $a^ib \in C$ for some $i$. If $n$ is even, then we must have $a^ib \sim a^{2n + i}b$ in $\c$ as N$[a^ib] = $N$[a^{2n +i}b] = \{e, a^{2n}, a^ib, a^{2n +i}b\}$ (cf. Lemma \ref{nbd-even}). As a result $a^{i}b$ is adjacent with only one element in $C$; a contradiction. In case of odd $n$, there exist $x, y \in C \; \cap $ N$(a^ib)$. Note that N$[a^ib] = $N$[a^{n  + i}] = $N$[a^{2n  + i}] = $N$[a^{3n  + i}] =  \{e, a^{n}, a^{2n}, a^{3n}, a^ib, a^{n + i}b,  a^{2n + i}b, a^{3n + i}b\}$ (cf. Lemma \ref{nbd-odd}) implies $x, y \in \{a^{n + i}b,  a^{2n + i}b, a^{3n + i}b\}$. Therefore, we have $x, y$ and $a^ib$ forms  a triangle in $C$; a contradiction. Thus, $\c$ does not contain any hole of odd length at least five.

Now assume that $C'$ is an antihole of length at least $5$ in $\c$, that is, we have a hole $y_1 \sim y_2 \sim \dots \sim y_{2l + 1} \sim y_1$, where $l \geq 2$, in $\overline{\c}$. Clearly, $y_i \notin Z(\c)$ for all $i$, where $1 \leq i \leq 2l + 1$. Suppose $y_i \in \langle a \rangle$ for some $i$. Then clearly $y_{i-1}, y_{i + 1} \in SD_{8n} \setminus \langle a \rangle$, otherwise $y_i \sim y_{i-1}$ and $y_i \sim y_{i +1}$ in $\c$; a contradiction. Further note that for $1 \leq j \leq 2l +1$ and $j \notin \{i-1, i, i+1\}$, we have  $y_j \in \langle a \rangle$. For instance, if  $y_j \in SD_{8n} \setminus \left(\langle a \rangle \cup \{y_{i -1}, y_i, y_{i + 1}\}\right)$ for some $j$, then $y_j \sim y_i$ in $\overline{\c}$ as $y_i \notin Z(\c)$ (see Lemmas \ref{nbd-even} and \ref{nbd-odd}); a contradiction. Therefore, there exists $y_j \in \langle a \rangle$ gives $y_j \sim y_{i-1}$ and $y_j \sim y_{i+1}$ in  $\overline{\c}$. As a result $\{y_j, y_{i-1}, y_i, y_{i+1}\}$ forms a cycle of length four in $\overline{\c}$; a contradiction. Thus,  $y_i \notin \langle a \rangle$ for all $i$.

 If $n$ is even, then $a^ib \sim a^jb$ for all $j \ne 2n + i$ in $\overline{\c}$ (see Lemma \ref{nbd-even}) implies $C'$ is not an antihole; again a contradiction. Now we assume that $n$ is odd. Let $y_1 = a^ib$ for some $i$. Then we have $y_3, y_4 \in $ N$(a^ib) = \{e, a^{n}, a^{2n}, a^{3n}, a^{n + i}b,  a^{2n + i}b, a^{3n + i}b\}$ gives $y_3 \nsim y_4$  in $C'$ (see  Lemma \ref{nbd-odd}); a contradiction. Thus, $\c$ does not contain any antihole of odd length at least five.
\end{proof}

\begin{theorem}
The commuting graph of $SD_{8n}$ is Hamiltonian if and only if $n \in \{1, 3\}$.
\end{theorem}

\begin{proof}
First suppose that $n$ is even. By Lemma \ref{nbd-even}, $\{e, a^{2n}\}$ is a vertex cut set in $\c$ so that by deletion of these vertices,  the connected components of the  subgraph induced by the vertices $V(\c )\setminus \{e, a^{2n}\}$ are $\{a^ib, a^{2n + i}b\}$ and $\langle a \rangle \setminus \{e, a^{2n}\}$ where $1 \leq i \leq 2n$. Therefore, it is impossible to construct Hamiltonian cycle in $\c$. Thus, $\c$ is not Hamiltonian graph when $n$ is even. Now we assume that $n$ is odd. By Lemma \ref{nbd-odd}, $\{e, a^n, a^{2n}, a^{3n}\}$ is a vertex cut set in $\c$ so that by deletion of these vertices, the connected components of the  subgraph induced by the vertices $V(\c )\setminus \{e, a^n, a^{2n}, a^{3n}\}$ are $\{a^ib, a^{n + i}b, a^{2n + i}b,  a^{3n + i}b \}$ and $\langle a \rangle \setminus \{e, a^n, a^{2n}, a^{3n}\}$ where $1 \leq i \leq n$ which are $n + 1$ in total. It follows that for the construction of Hamiltonian cycle in $\c$, we required at least $n + 1$ element from the vertex cut set. Thus, for $n > 3$, $\c$ is not Hamiltonian graph. For $n = 1, 3$, in view of Lemma \ref{nbd-odd},  we have deg$(x) \geq \frac{|V(\c)|}{2}$ for all $x \in SD_{8n}$. Thus, by {\cite[Theorem 7.2.8]{b.West}},  $\c$ is Hamiltonian.
\end{proof}
\subsection{Laplacian spectrum}
In this subsection, we investigate the Laplacian spectrum  of $\c$. Consequently, we provide the number of spanning trees of $\c$. For a finite simple undirected graph $\Gamma$ with vertex set $V(\Gamma) = \{v_1, v_2, \ldots, v_n\}$, the \emph{adjacency matrix} $A(\Gamma)$ is the $n\times n$ matrix with $(i, j)th$ entry is $1$ if $v_i$ and $v_j$ are adjacent and $0$ otherwise. We denote the diagonal matrix $D(\Gamma) = {\rm diag}(d_1, d_2, \ldots, d_n)$ where $d_i$ is the degree of the vertex $v_i$ of $\Gamma$, $i = 1, 2, \ldots, n$. The \emph{Laplacian matrix} $L(\Gamma)$ of $\Gamma$ is the matrix $D(\Gamma) - A(\Gamma)$. The matrix $L(\Gamma)$ is symmetric and positive semidefinite, so that its eigenvalues are real and non-negative. Furthermore, the sum of each row (column) of $L(\Gamma)$ is zero.  Recall that the \emph{characteristic polynomial} of $L(\Gamma)$ is denoted by $\Phi(L(\Gamma), x)$. The eigenvalues of $L(\Gamma)$ are called the \emph{Laplacian eigenvalues} of $\Gamma$ and it is denoted by $\lambda_1(\Gamma) \geq \lambda_2(\Gamma) \geq \cdots \geq \lambda_n(\Gamma) = 0$. Now let $\lambda_{n_1}(\Gamma) \geq \lambda_{n_2}(\Gamma) \geq \cdots \geq \lambda_{n_r}(\Gamma) = 0$ be the distinct eigenvalues of $\Gamma$ with multiplicities $m_1, m_2, \ldots, m_r$, respectively. The \emph{Laplacian spectrum} of $\Gamma$, that is, the spectrum  of $L(\Gamma)$ is represented as $\displaystyle \begin{pmatrix}
\lambda_{n_1}(\Gamma) & \lambda_{n_2}(\Gamma) & \cdots& \lambda_{n_r}(\Gamma)\\
 m_1 & m_2 & \cdots & m_r\\
\end{pmatrix}$. 
We denote the matrix $J_n$ as the square matrix of order $n$ having all the entries as $1$ and $I_n$ is the identity matrix of order $n$. In the following theorem, we obtain the characteristic polynomial of $L(\c)$.

\begin{theorem}\label{lap-even-SD_8n}  
For even $n$, the characteristic polynomial of the Laplacian matrix of $\c$ is given by
\[\Phi(L(\c), x) =x(x -8n)^2 (x -4)^{2n}(x -2)^{2n}(x-4n)^{4n -3}.\] 
\end{theorem}

\begin{proof}
The Laplacian matrix $L(\c)$ is the $8n \times 8n$ matrix given below, where the rows and columns are indexed in order by the vertices $e = a^{4n}, a^{2n}, a, a^2, \ldots, a^{2n - 1},a^{2n + 1}, a^{2n + 2}, \ldots, a^{4n - 1}$ and then $ab, a^2b,  \ldots,  a^{4n}b$.
	\[L(\c)  = \displaystyle \begin{pmatrix}
	8n-1 & -1 & -1 & -1& \cdots \cdots & -1  &-1 & \cdots \cdots &-1  \\
	-1 & 8n - 1 & -1 & -1& \cdots \cdots & -1  &-1  & \cdots \cdots &-1  \\
	-1  & -1&   &  &  &  &  &  &    \\
	-1  & -1 &   & A &  &    &  &  \mathcal O  & & \\
	\; \; \vdots & \; \;\vdots& & &   &  &  &  &  & \\ 
	\; \; \vdots& \; \; \vdots& & & &  & &  &  & \\ 
	-1  & -1 &   &  &  &  & &  & &     \\ 
	-1  & -1 &   &  &  &  & &  & &     \\
	
	\; \; \vdots& \; \; \vdots& & \mathcal O'& &  & & B &  & \\ 
	\; \; \vdots& \; \; \vdots& & & &  & &  &  & \\ 
	-1  & -1&   & &   &  & & & &  \\
	\end{pmatrix}\]
	where $A = 4n I_{4n -2} - J_{4n -2}$, $B = \displaystyle \begin{vmatrix}
	3I_{2n} & -I_{2n}\\
	-I_{2n} & 3I_{2n}
	\end{vmatrix}$, $\mathcal O$ is the zero matrix of size $(4n - 2) \times (4n)$ and $\mathcal O'$ is the transpose matrix of $\mathcal O$. Then the characteristic polynomial of $L(\c)$ is 
	
	\[\Phi(L(\c), x)  = \displaystyle \begin{vmatrix}
	x - (8n-1) & 1 & 1 & 1& \cdots \cdots & 1  & 1 & \cdots \cdots &1  \\
	1 &x - ( 8n - 1) & 1 & 1& \cdots \cdots & 1  &1  & \cdots \cdots &1  \\
	1  & 1&   &  &  &  &  &  &    \\
	1  & 1 &   & (xI_{4n -2} - A) &  &    &  &  \mathcal O  & & \\
	\vdots & \vdots& & &   &  &  &  &  & \\ 
	\vdots& 	 \vdots& & & &  & &  &  & \\ 
	1  & 1 &   &  &  &  & &  & &     \\ 
	1  & 1 &   &  &  &  & &  & &     \\
	
	\vdots&  \vdots& & \mathcal O'& &  & &(xI_{4n}-B) &  & \\ 
	\vdots&  \vdots& & & &  & &  &  & \\ 
	1  & 1&   & &   &  & & & &  \\
	\end{vmatrix}.\]
	
	Apply row operation $R_1 \rightarrow (x -1)R_1 - R_2 - \cdots -R_{8n}$ and then expand by using first row, we get\\
	
	$\Phi(L(\c), x) = \frac{x(x -8n)}{(x - 1)}\displaystyle \begin{vmatrix}
	x - (8n-1) & 1 & 1& \cdots \cdots & 1  & 1 & \cdots \cdots &1\\
	1  &   &  &  &    &  &    & & \\
	\vdots &  & (xI_{4n -2} - A)&   &  &  & \mathcal O &  & \\ 
	\vdots &   & & &  & &  &  & \\ 
	1  &   &  &  &  & &  & &     \\ 
	1  &    &  &  &  & &  & &     \\
	
	\vdots&   & \mathcal O'& &  & &(xI_{4n}-B) &  & \\ 
	\vdots&   & & &  & &  &  & \\ 
	1 &   & &   &  & & & &  \\
	\end{vmatrix}.$\\
	
	Again, apply row operation $R_1 \rightarrow (x -2)R_1 - R_2 - R_3 -\cdots - R_{8n -1}$ and then expand by using first row, we get
	\[\Phi(L(\c), x) = \frac{x(x -8n)^2}{(x-2)}  \displaystyle \begin{vmatrix}
	xI_{4n -2} - A & \mathcal O\\
	\mathcal O' & xI_{4n} - B
	\end{vmatrix}.\] By using Schur's decomposition theorem \cite{b.spectra}, we have \[\Phi(L(\c), x) = \frac{x(x -8n)^2}{(x-2)} |xI_{4n -2} - A|\cdot |xI_{4n}-B|.\] Clearly, 
	$|xI_{4n} - B| = \displaystyle \begin{vmatrix}
	(x - 3)I_{2n} & I_{2n}\\
	I_{2n} & (x -3)I_{2n}
	\end{vmatrix}$. Again by using Schur's decomposition theorem, we obtain \[|xI_{4n} - B| = |(x- 3)I_{2n}| |(x- 3)I_{2n} - \frac{1}{(x - 3)}I_{2n}| = (x -4)^{2n}(x -2)^{2n}.\] Now we obtain $|xI_{4n -2} - A| = |xI_{4n -2} -(4n I_{4n -2} - J_{4n -2})|$. It is easy to compute the characteristic polynomial of the matrix $J_{4n -2}$ is $x^{4n -3}(x-4n + 2)$. It is well known that if $f(x) = 0$ is any polynomial and $\lambda$ is an eigenvalue of the matrix $P$, then $f(\lambda)$ is an eigenvalue of the matrix $f(P)$. Consequently, the eigenvalues of the matrix $A$ are $4n$ and $2$. Note that if $x$ is an eigenvector of $J_n$ corresponding to the eigenvalue $0$, then $x$ is also an eigenvector of the matrix $A$ corresponding to eigenvalue $4n$. since dimension of the null space of $J_{4n -2}$ is  $4n - 3$  so that the multiplicity of the eigenvalue $4n$ in the characteristic polynomial of the matrix $A$ is $4n - 3$. Thus, $|xI_{4n -2} - A| = (x-4n)^{4n -3} (x-2)$ and hence the result holds.   
\end{proof}

\begin{corollary}
For even $n$, the Laplacian spectrum of $\c$ is given by 
\[\displaystyle \begin{pmatrix}
0 & 2 &  4 & 4n & 8n\\
 1 & 2n & 2n & 4n -3 & 2\\
\end{pmatrix}.\]
\end{corollary}

By {\cite[Corollary 4.2]{a.Mohar}}, we have the following corollary. 

\begin{corollary}
 For even $n$, the number of spanning trees of $\c$ is $2^{14n - 3}n^{4n - 2}$.
\end{corollary}
 
\begin{theorem}\label{lap-odd-SD_8n}
For odd $n$, the characteristic polynomial of the Laplacian matrix of $\c$ is given by \[\Phi(L(\c), x) =x(x -8n)^4 (x -4)^{n}(x -8)^{3n}(x-4n)^{4n -5}.\]
\end{theorem}
\begin{proof}
The Laplacian matrix $L(\c)$ is the $8n \times 8n$ matrix given below, where the rows and columns are indexed by the vertices $e = a^{4n}, a^{3n},  a^{2n},  a^{n}, a, a^2, \ldots, a^{n - 1},a^{n + 1},  a^{n + 2}, \ldots, a^{2n - 1}, a^{2n + 1}, \\  a^{2n + 2}, \ldots, a^{3n - 1}$, $a^{3n + 1},  a^{3n + 2}, \ldots, a^{4n - 1}$ and then  $ab, a^2b,  \ldots,  a^{4n}b$.
	\[L(\c)  = \displaystyle \begin{pmatrix}
	8n-1 & -1 & -1 & -1& \cdots \cdots & -1  &-1 & \cdots \cdots &-1  \\
	-1 & 8n - 1 & -1 & -1& \cdots \cdots & -1  &-1  & \cdots \cdots &-1  \\
	-1 &  - 1 &8n -1 & -1& \cdots \cdots & -1  &-1  & \cdots \cdots &-1  \\
	-1 & - 1 & -1 &8n -1& \cdots \cdots & -1  &-1  & \cdots \cdots &-1  \\
	-1  & -1& -1  & -1 &  &  &  &  &    \\
		-1  & -1& -1  & -1 &  &  &  &  &    \\
	-1  & -1 &  -1 & -1 & A &    &  &  \mathcal O  & & \\
	\; \; \vdots & \; \;\vdots& \; \;\vdots & \; \;\vdots &   &  &  &  &  & \\ 
	-1  & -1 & -1  & -1 &  &  & &  & &     \\ 
	\; \; \vdots& \; \; \vdots&\; \;\vdots &\; \;\vdots & \mathcal O'& &  & B &  &  \\ 
	-1  & -1& -1  & -1 &   &  & & & &  \\
	\end{pmatrix}\]
	where $A = 4nI_{(4n - 4)} - J_{(4n - 4)}$, $B = \displaystyle \begin{pmatrix}
	7I_{n} & -I_{n} & -I_{n} & -I_{n}\\
	-I_{n} & 7I_{n} & -I_{n} & -I_{n}\\
	-I_{n} & -I_{n} & 7I_{n} & -I_{n}\\
	-I_{n} & -I_{n} & -I_{n} & 7I_{n}\\
	\end{pmatrix}$, $\mathcal O$ and $\mathcal O'$ are defined in 
\vspace{0.3cm}	
	Theorem \ref{lap-even-SD_8n}. Then the characteristic polynomial of $L(\c)$ is 
	
	\[\Phi(L(\c), x)  = \displaystyle \begin{vmatrix}
	x-(8n-1) & 1 & 1 & 1& \cdots & 1  &1 & \cdots 1  \\
	1 &x-(8n-1) & 1 & 1& \cdots  & 1  &1 & \cdots  1  \\
	1 &1 &x-(8n-1) & 1& \cdots  & 1  &1 & \cdots  1  \\
	1 &1 & 1 &x-(8n-1)& \cdots  & 1  &1 & \cdots  1  \\
	1  & 1& 1  & 1 &  &  &  &  &    \\
	1  & 1 &  1 & 1 & xI- A &    &  &  \mathcal O  & \\
	\vdots & \vdots& \vdots & \vdots &   &  &  &  &  \\ 
	\vdots&  \vdots& \vdots& \vdots & &  & &  &  \\ 
	1  & 1 & 1  & 1 &  &  & &  &     \\ 
	1  & 1 &  1 & 1 &  &  & &  &      \\
	
	\vdots&  \vdots&\vdots & \vdots & \mathcal O'& &  & xI - B &    \\ 
	\vdots&  \vdots&  \vdots &  \vdots & &  & &  &  \\ 
	1  & 1& 1  & 1 &   &  & & & \\
	\end{vmatrix}\]
	
	Apply the following row operations consecutively
	\begin{itemize}
		\item $R_1 \rightarrow (x -1)R_1 - R_2  - \cdots - R_{8n}$
		\item $R_2 \rightarrow (x -2)R_2 - R_3 - \cdots - R_{8n}$
		\item $R_3 \rightarrow (x -3)R_3 - R_4 - \cdots - R_{8n}$
		\item  $R_4 \rightarrow (x -4)R_4 - R_5 - \cdots - R_{8n}$
	\end{itemize}	
	
	and then expand, we get
	
	\[\Phi(L(\c), x) = \frac{x(x -8n)^4}{(x-4)}  \displaystyle \begin{vmatrix}
	xI - A & \mathcal O\\
	\mathcal O' & xI - B
	\end{vmatrix} = \frac{x(x -8n)^4}{(x-4)} |xI-A||xI-B|.\]
	By the similar argument used in the proof of Theorem \ref{lap-even-SD_8n}, we obtain $|xI - A| = (x-4n)^{4n -5} (x-4)$. To get \[|xI-B| = \displaystyle \begin{vmatrix}
	(x-7)I_{n} & I_{n} & I_{n} & I_{n}\\
	I_{n} & (x-7)I_{n} & I_{n} & I_{n}\\
	I_{n} & I_{n} & (x-7)I_{n} & I_{n}\\
	I_{n} & I_{n} & I_{n} & (x-7)I_{n}\\
	\end{vmatrix},\]
\\	
	apply the following row operations consecutively $R_i \rightarrow (x -5)R_i - R_{i +1}  - \cdots - R_{4n}$ where $1 \leq i \leq n$ and then on solving, we get \[|xI-B| = \frac{(x -4)^n (x-8)^n}{(x-5)^n} \displaystyle \begin{vmatrix}
	(x-7)I_{n} & I_{n} & I_{n}\\
	I_{n} & (x-7)I_{n} & I_{n}\\
	I_{n} & I_{n} & (x-7)I_{n}\\
	\end{vmatrix}.\]

Again apply the following row operations consecutively
\begin{itemize}
\item For $1 \leq i \leq n$  $R_i \rightarrow (x -6)R_i - R_{i +1}  - \cdots - R_{3n}$
\item For $n + 1 \leq i \leq 2n$  $R_i \rightarrow (x -7)R_i - R_{i +1}  - \cdots - R_{3n}$
\end{itemize}
and then expand,  we obtain \[|xI-B| = \frac{(x -4)^n (x-8)^{3n}}{(x-7)^n} |(x-7)I_n| = (x -4)^n (x-8)^{3n}.\]
Thus, the result holds.
\end{proof}

\begin{corollary}
For odd $n$, the Laplacian spectrum of $\c$ is given by 
\[\displaystyle \begin{pmatrix}
0 & 4 &  8 & 4n & 8n\\
 1 & n & 3n & 4n -5 & 4\\
\end{pmatrix}.\]
\end{corollary}

By {\cite[Corollary 4.2]{a.Mohar}}, we have the following corollary.

\begin{corollary}
For odd $n$, the number of spanning trees of $\c$ is $2^{19n - 1} n^{4n - 2}$.
\end{corollary}

\subsection{Resolving Polynomial} In this subsection, we obtain the resolving polynomial of $\c$. First, we recall some of the basic definitions and necessary results.
For $z$ in $\Gamma$, we say that $z$ \emph{resolves} $u$ and $v$ if $d(z, u) \ne d(z, v)$. A subset $U$ of $V(\Gamma)$ is a \emph{resolving set} of $\Gamma$ if every pair of vertices of $\Gamma$ is resolved by some vertex of $U$. The least cardinality of a resolving set of $\Gamma$ is called the \emph{metric dimension} of $\Gamma$ and is denoted by $\operatorname{dim}(\Gamma)$. An $i$-\emph{subset} of $V(\Gamma)$ is a subset of $V(\Gamma)$ of cardinality $i$. Let $\mathcal R(\Gamma, i)$ be the family of resolving sets which are $i$-subsets and $r_i = |\mathcal R(\Gamma, i)|$. Then we define the \emph{resolving polynomial} of a graph $\Gamma$ of order $n$, denoted by $\beta(\Gamma, x)$ as $\beta(\Gamma, x) = \mathop{\sum}_{i= dim(\Gamma)}^{n} r_ix^i$. The sequence $(r_{dim(\Gamma)}, r_{dim(\Gamma) + 1}, \ldots, r_n)$ of coefficients of $\beta(\Gamma, x)$ is called the \emph{resolving sequence}.  Two distinct vertices $u$ and $v$ are said to be \emph{true twins} if N$[u] = $N$[v]$. Two distinct vertices $u$ and $v$ are said to be \emph{false twins} if N$(u) = $N$(v)$. If $u$ and $v$ are true twins or false twins then $u$ and $v$ are \emph{twins}. A set $U \subseteq V(\Gamma)$ is said to be a \emph{twin-set} in $\Gamma$ if $u, v$ are twins for every pair of distinct pair of vertices $u, v \in U$. In order to obtain the resolving polynomial $\beta(\c, x)$, the following results will be useful. 

\begin{remark}[{\cite[Remark 3.3]{a.Ali-2016}}]\label{r.twin-set}
If $U$ is twin-set in a connected graph $\Gamma$ of order $n$ with $|U| = l \geq 2$, then every resolving set for $\Gamma$ contains at least $l-1$ vertices of $U$.
\end{remark}

\begin{proposition}[{\cite[Proposition 3.5]{a.Ali-2016}}]\label{p.resolving-SD_8n}
Let $\Gamma$ be a connected graph of order $n$. Then the only resolving set of cardinality $n$ is the set $V(\Gamma)$ and a resolving set of cardinality $n-1$ can be chosen $n$ possible different ways.
\end{proposition}

\begin{proposition}\label{p.metric-SD_8n}
The metric dimension of $\c$ is given below: \[{\rm dim}(\c) = \left\{ \begin{array}{ll}
6n-2 & \text{when $n$ is even };\\
7n - 2 & \text{otherwise.}\end{array}\right. \]
\end{proposition}

\begin{proof}
First we assume that $n$ is even. In view of Lemma \ref{nbd-even}, we get twin-sets $ \langle a \rangle \setminus \{e, a^{2n}\}, \{e, a^{2n}\}$ and $\{a^ib, a^{2n + i}b\}$ where $1 \leq i \leq 2n$.  By Remark \ref{r.twin-set}, any resolving set in $\c$ contains at least $6n-2$ vertices. Now we provide a resolving set of size $6n-2$. By Lemma \ref{nbd-even}, one can verify that the set $R_{\rm even} = \{a^ib : 1\leq i \leq 2n \} \cup \{a^i : i \ne 1, 2n \}$ is a resolving set of size $6n - 2$. Consequently, ${\rm dim}(\c) = 6n - 2$. We may now suppose that $n$ is odd. By Lemma \ref{nbd-odd}, note that $\langle a \rangle \setminus \{e, a^n, a^{2n}, a^{3n}\}, \{e, a^n a^{2n}, a^{3n}\}$ and $\{a^ib, a^{n +i}b, a^{2n + i}b, a^{n +3i}b\}$, where $1 \leq i \leq n$, are twin sets in $\c$. In view of Remark \ref{r.twin-set}, any resolving set in $\c$ contains at least $7n-2$ vertices. Further, it is routine to verify that the set $R_{\rm odd} = \{a^ib, a^{n + i}b, a^{2n + i}b : 1 \leq i \leq n\} \cup \{a^i : i \ne 1, 2n \}$ is a resolving set of size $7n-2$. Thus, ${\rm dim}(\c) = 7n - 2$. 
\end{proof}

\begin{theorem}
For even $n$, the resolving polynomial of $\c$ is given below:
\[\beta(\c, x) = x^{8n} + 8n x^{8n-1} + 2^{2n + 2}(2n - 1) x^ {6n-2} + \mathop{\sum}_{i= 6n-1}^{8n - 2} r_ix^i,\]
where $r_i = 2^{8n - i} \left\{ \binom{2n + 1}{8n - i} + (2n - 1) \binom{2n + 1}{8n - i -1} \right\}$ for $6n -1 \leq i \leq 8n -2$.
\end{theorem}

\begin{proof}
In view of Proposition \ref{p.metric-SD_8n}, we have dim$(\c) = 6n - 2$. It is sufficient to find the resolving sequence $(r_{6n - 2}, r_{6n -1}, \ldots, r_{8n -2},  r_{8n -1}, r_{8n})$. By the proof of Proposition \ref{p.metric-SD_8n}, any resolving set $R$ satisfies the  following:
\begin{itemize}
\item $|R \cap (\langle a \rangle \setminus \{e, a^{2n}\})| \geq 4n - 3$;
\item $|R \cap  \{e, a^{2n}\}| \geq 1$;
\item $|R \cap \{a^ib, a^{2n + i}b\}| \geq 1$ where $1 \leq i \leq 2n$.
\end{itemize}
For $|R| = i\geq 6n -2$, there exist $v_1, v_2, \ldots, v_{8n-i} \in SD_{8n} \setminus R$. Therefore we have one of the following:
\begin{enumerate}[(i)]
\item $v_j \in \langle a \rangle \setminus \{e, a^{2n}\}$ for some $j$  and \\ $ v_1, v_2, \ldots, v_{j-1}, v_{j+1}, v_{j+2}, \ldots v_{8n-i} \in \left(\displaystyle \bigcup\limits_{i = 1}^{2n} \{a^ib, a^{2n +i}b\}\right)\cup \{e, a^{2n}\}$.
\item $v_1, v_2, \ldots, v_{8n-i} \in \left(\displaystyle \bigcup\limits_{i = 1}^{2n} \{a^ib, a^{2n +i}b\}\right)\cup \{e, a^{2n}\}$.
\end{enumerate}  

For $i = 6n -2$, (ii) does not hold so $v_j \in \langle a \rangle \setminus \{e, a^{2n}\}$ and \\ $ v_1, v_2 \ldots, v_{j-1}, v_{j+1}, v_{j+2}, \ldots, v_{8n-i} \in \left(\displaystyle \bigcup\limits_{i = 1}^{2n} \{a^ib, a^{2n +i}b\}\right)\cup \{e, a^{2n}\}$. Therefore, we obtain $r_{6n-2} = 2^{2n + 1}(4n - 2)$. Now for fixed $i, \; 6n - 1 \leq i \leq 8n - 2$, we get $r_i  = 2^{8n - i} \left\{ \binom{2n + 1}{8n - i} + (2n - 1) \binom{2n + 1}{8n - i -1}  \right\}$.  By Proposition \ref{p.resolving-SD_8n}, $r_{8n-1}= 8n$ and $r_{8n} = 1$.  
\end{proof}

\begin{theorem}
For odd $n$, the resolving polynomial of $\c$ is given below:
\[\beta(\c, x) = x^{8n} + 8n x^{8n - 1} + 2^{2n + 4}(n - 1)x^{7n-2} +\mathop{\sum}_{i= 7n-1}^{8n -  2} r_ix^i,\]
where $r_i = 2^{16n - 2i} \left\{ \binom{n + 1}{8n - i} + (n - 1) \binom{n + 1}{8n - i -1}  \right\}$ for $7n -1 \leq i \leq 8n -2$.
\end{theorem}

\begin{proof}
In view of Proposition \ref{p.metric-SD_8n}, we have dim$(\c) = 7n - 2$. It is sufficient to find the resolving sequence $(r_{7n - 2}, r_{7n -1}, \ldots, r_{8n -2},  r_{8n -1}, r_{8n})$. By the proof of Proposition \ref{p.metric-SD_8n}, any resolving set $R$ satisfies the  following:
\begin{itemize}
\item $|R \cap (\langle a \rangle \setminus \{e, a^n, a^{2n}, a^{3n}\})| \geq 4n - 5$;
\item $|R \cap  \{e, a^n, a^{2n}, a^{3n}\}| \geq 3$;
\item $|R \cap \{a^ib, a^{n +  i}b,  a^{2n + i}b, a^{3n + i}b\}| \geq 3$ where $1 \leq i \leq n$.
\end{itemize}
For $|R| = i\geq 7n -2$, there exist $v_1, v_2, \ldots, v_{8n-i} \in SD_{8n} \setminus R$. Therefore we have one of the following
\begin{enumerate}[(i)]
\item $v_j \in \langle a \rangle \setminus \{e, a^n, a^{2n}, a^{3n}\}$ for some $j$  and \\ $ v_1, v_2, \ldots, v_{j-1}, v_{j+1}, v_{j+2}, \ldots v_{8n-i} \in \left(\displaystyle \bigcup\limits_{i = 1}^{n} \{a^ib, a^{n +i}b, a^{2n +i}b, a^{3n + i}b\}\right)\cup \{e, a^n, a^{2n}, a^{3n}\}$.
\item $v_1, v_2, \ldots, v_{8n-i} \in \left(\displaystyle \bigcup\limits_{i = 1}^{n} \{a^ib, a^{n +i}b, a^{2n +i}b, a^{3n + i}b\}\right)\cup \{e, a^n, a^{2n}, a^{3n}\}$.
\end{enumerate}  
	
For $i = 7n -2$, (ii) does not hold so $v_j \in \langle a \rangle \setminus \{e, a^{2n}\}$ and \\ $ v_1, v_2, \ldots, v_{j-1}, v_{j+1}, v_{j+2}, \ldots v_{8n-i} \in \left(\displaystyle \bigcup\limits_{i = 1}^{n} \{a^ib, a^{n +i}b, a^{2n +i}b, a^{3n + i}b\}\right)\cup \{e, a^n, a^{2n}, a^{3n}\}$. Therefore, we have $r_{7n-2} = 4^{n + 1}(4n - 4)$. Now for fixed $i, \; 7n - 1 \leq i \leq 8n - 2$, we get \[r_i = 2^{16n - 2i} \left\{ \binom{n + 1}{8n - i} + (n - 1) \binom{n + 1}{8n - i -1}  \right\}.\]  By Proposition \ref{p.resolving-SD_8n}, $r_{8n-1}= 8n$ and $r_{8n} = 1$.  
\end{proof}

\subsection{Detour distance properties}
In this subsection,  we study the detour distance properties of $\c$ viz.  detour radius, detour eccentricity, detour degree, detour degree sequence and detour distance degree sequence of each vertex. 

\begin{theorem} \label{detour-ecentricity}
In  $\c$, we have for each $x \in Z(SD_{8n})$, 
\[ecc_D(x) =\left\{ \begin{array}{ll}
4n + 1 & \text{when $n$ is even};\\
4n + 11 & \text{when $n$ is odd.}\end{array}\right.\]

and for each $x \in SD_{8n} \setminus Z(SD_{8n})$,  \[ecc_D(x) =\left\{ \begin{array}{ll}
4n + 3 & \text{when $n$ is even};\\
4n + 15 & \text{when $n$ is odd.}\end{array}\right.\]
\end{theorem}

\begin{proof}
 We split our proof in two cases depend on $n$.\\
\textbf{Case 1.} $n$ is even. First note that $x \sim y$ for $x \in Z(SD_{8n})$ and $y \in SD_{8n} \setminus \{x\}$; $x' \sim y'$ for all distinct $x', y' \in \langle a \rangle \setminus \{e, a^{2n}\}$; for each $1 \leq i \leq 4n$, $a^ib \sim a^{2n + i}b$, $a^ib \nsim a^jb$ for all $j \ne 2n + i$, $a^ib \nsim a^j$ for all $a^j \in \langle a \rangle \setminus \{e, a^{2n}\}$. Thus we have (i) for each $x \in Z(SD_{8n})$, there is a $x$ - $y$ detour of length $4n -1$ for all $y \in Z(SD_{8n}) \setminus \{x\}$; a $x$ - $ y$ detour of length $4n +1$ for all $y \in SD_{8n} \setminus Z(SD_{8n})$ as $Z(SD_{8n}) = \{e, a^{2n}\}$; (ii) for each $1 \leq i \leq 4n$, there is a $a^ib$ - $a^{2n+i}b$ detour of length $4n +1$; for distinct $1 \leq i, j \leq 4n$ and $j \ne 2n +i$, a $a^ib$ - $a^jb$  detour of length $4n +3$; for each $1 \leq i \leq 4n$ and for each $a^j \in \langle a \rangle \setminus Z(SD_{8n})$, a $a^ib$ - $a^j$ detour  of length $4n +3$; and (iii) for distinct $1 \leq i, j < 4n$ and $i, j \ne 2n$, there is a $a^i$ - $a^j$ detour of length $4n + 1$.\\
\textbf{Case 2.} $n$ is odd. First note that $x \sim y$ for $x \in Z(SD_{8n})$ and $y \in SD_{8n} \setminus \{x\}$; $x' \sim y'$ for all distinct $x', y' \in \langle a \rangle \setminus \{e, a^n a^{2n}, a^{3n}\}$; for each $1 \leq i \leq n$ and for each $j \in \{n +i, 2n +i, 3n +i\}$, $a^ib \sim a^{j}b$; for each $1 \leq i \leq 4n$, $a^ib \nsim a^jb$ for all $j \notin \{n +i,2n +i, 3n +i\}$, $a^ib \nsim a^j$ for all $a^j \in \langle a \rangle \setminus \{e, a^n, a^{2n}, a^{3n}\}$. Thus we have (i) for each $x \in Z(SD_{8n})$, there is a $x$ - $y$ detour of length $4n +7$ for all $y \in Z(SD_{8n}) \setminus \{x\}$; a $x$ - $y$ detour of length $4n +11$ for all $y \in SD_{8n} \setminus Z(SD_{8n})$; (ii) for each $1 \leq i \leq 4n$, there is a $a^ib$ - $a^{j}b$  detour of length $4n +11$ for all $j \notin \{n +i,2n +i, 3n +i\}$; for distinct $1 \leq i \leq 4n$ and $j \notin \{n +i,2n +i, 3n +i\}$, a $a^ib$ - $a^{j}b$ detour  of length $4n +15$; for each $1 \leq i \leq 4n$ and for each $a^j \in \langle a \rangle \setminus Z(SD_{8n})$, a $a^ib$ - $a^j$ detour of length $4n +15$; and (iii) for distinct $1 \leq i, j < 4n$ and $i, j \notin \{n, 2n, 3n, 4n\}$, there is a $a^i$ - $a^j$ detour of length $4n + 11$.
\end{proof}

By the definition of $rad_D(\c)$ and $diam_D(\c)$, we have the following corollary.

\begin{corollary}\label{Detour-radius-diameter}
In $\c$, we have
\begin{enumerate}[\rm(i)]
\item $rad_D(\c) = \left\{ \begin{array}{ll}
4n + 1 & \text{if $n$ is even };\\
4n + 11 & \text{if $n$ is odd.}\end{array}\right.$

\vspace{0.2cm}
		
\item $diam_D(\c) = \left\{ \begin{array}{ll}
4n + 3 & \text{if $n$ is even };\\
4n + 15 & \text{if $n$ is odd.}\end{array}\right.$
\end{enumerate}
\end{corollary}

The \emph{detour degree} $d_D(v)$ of $v$ is the number $|D(v)|$, where $D(v) = \{u \in V(\Gamma) : d_D(u, v) = ecc_D(v) \}$. The average detour degree $(D_{av}(G))$ of a
graph $\Gamma$ is the quotient of the sum of the detour degrees of all the vertices
of $\Gamma$ and the order of $G$. The detour degrees of the vertices of a graph
written in non-increasing order is said to be the detour degree sequence of graph
$\Gamma$, denoted by $D(\Gamma)$. For a vertex $x \in V(\Gamma)$, we denote $D_i(x)$ be the number of vertices at a detour distance $i$ from the vertex $x$, then the sequence $D_0(x), D_1(x), D_2(x), \ldots, D_{ecc_D(x)}(x)$ is called \emph{detour distance degree sequence of a vertex x}, denoted by $dds_D(x)$. In the remaining part of this paper, $(a^r, b^s, c^t)$ denote $a$ occur $r$ times, $b$ occur $s$ times and $c$ occur $t$ times in the sequence. Now we have the following remark.

\begin{remark}[{\cite[Remark 2.6]{a.Ali-2016}}] In a graph $\Gamma$, we have
\begin{enumerate}[\rm (i)]
\item $D_0(v) = 1$ and $D_{ecc_D}(v) = d_D(v)$.
\item The length of sequence $dds_D(v)$ is one more than the detour eccentricity of $v$.
\item $\displaystyle \sum\limits_{i = 0}^{ecc_D(v)} D_i(v) = |\Gamma|$.
\end{enumerate}
\end{remark}

\begin{proposition}\label{Detour-degree}
In $\c$, we have for each $x \in Z(SD_{8n})$

\[d_D(x) =  \left\{ \begin{array}{ll}
8n - 2 & \text{when $n$ is even};\\
8n -4 & \text{when $n$ is odd;}\end{array}\right.\]

for each $1 \leq i \leq 4n$,
\[d_D(a^ib) =  \left\{ \begin{array}{ll}
8n - 4 & \text{when $n$ is even};\\
8n -8 & \text{when $n$ is odd;}\end{array}\right.\]

and for each $x \in \langle a \rangle \setminus Z(SD_{8n})$, $d_D(x) = 4n$.
\end{proposition}

\begin{proof} 
Let $x \in Z(SD_{8n})$.  In view of Theorem \ref{detour-ecentricity}, $ecc_D(x) = 4n +1$ when $n$ is even. Otherwise $ecc_D(x) = 4n +11$. In each case, by the proof of Theorem \ref{detour-ecentricity}, one can observe that  $D(x) = SD_{8n} \setminus Z(SD_{8n})$. Similar to $x \in Z(SD_{8n})$, for $x \in \langle a \rangle \setminus Z(SD_{8n})$ we obtain $D(x) = SD_{8n} \setminus \langle a \rangle$ (cf. Theorem \ref{detour-ecentricity}). Now let $x = a^ib$ for some $i$, where $1 \leq i \leq 4n$. Similar to $x \in Z(SD_{8n})$, when $n$ is even, we get
\[D(a^ib) = \left(\{a^jb : 1\leq j \leq 4n\} \setminus \{a^ib, a^{2n +i}b\}\right) \cup \left(\langle a\rangle \setminus Z(SD_{8n}) \right).\]
and for odd $n$,
\[D(a^ib) = \left(\{a^jb : 1\leq j \leq 4n\} \setminus \{a^{n+i}b, a^{2n +i}b, a^{3n +i}b, a^{4n + i}b\}\right) \cup \left(\langle a\rangle \setminus Z(SD_{8n}) \right).\]
\end{proof}

\begin{corollary} In $\c$, we have

\begin{enumerate} [\rm (i)]
\item \[D(\c) = \left\{ \begin{array}{ll}
\left((4n)^{4n -2},(8n - 4)^{4n}, (8n-2)^2\right) & \text{if $n$ is even};\\
\left((4n)^{4n -4}, (8n-8)^{4n}, (8n-4)^4\right) &  \text{if $n$ is odd.}  \end{array}\right.\]

\item \[D_{av}(\c) = \left\{ \begin{array}{ll}
2^6n(2n - 1)^2(4n -1) & \text{if $n$ is even};\\
2^{10}n(n-1)^2(2n-1) &  \text{if $n$ is odd.}
 \end{array}\right.\]
\end{enumerate}
\end{corollary}

\begin{theorem}
In $\c$, we have
\[dds_D(\c) = \left\{ \begin{array}{ll}
(1,0^{4n-2}, 1, 0, 8n-2)^2, (1,0^{4n}, 4n -1, 0, 4n)^{4n-2}, (1, 0^{4n}, 3, 0, 8n - 4)^{4n}\\ \text{if $n$ is even };\\ 
\vspace{0.1cm}\\

(1,0^{4n+6}, 3, 0^3, 8n-4)^4, (1,0^{4n +10}, 4n -1, 0^3, 4n)^{4n -4}, (1, 0^{4n +10}, 7, 0^3,8n - 8)^{4n}\\ \text{if $n$ is odd.}\end{array}\right. \]
\end{theorem}

\begin{proof}
\textbf{Case 1:} $n$ is even. For $x \in Z(SD_{8n})$, by the proof of Theorem \ref{detour-ecentricity} (\textbf{Case 1}), we have  $ecc_D(x) = 4n +1$  so $dds_D(x) = (1,\underbrace{0, 0, \ldots, 0}_\text{$4n -2$},1, 0, 8n-2)$. For $x \in SD_{8n} \setminus Z(SD_{8n})$, again by the proof of Theorem \ref{detour-ecentricity} (\textbf{Case 1}), we have  $ecc_D(x) = 4n +3$. Thus \[dds_D(x) = \left\{ \begin{array}{ll} 
(1,\underbrace{0, 0, \ldots, 0}_\text{$4n$},4n-1, 0, 4n) & \text{if $x \in \langle a \rangle \setminus Z(SD_{8n})$}\\

(1,\underbrace{0, 0, \ldots, 0}_\text{$4n -2$},3, 0, 8n-4) & \text{if $x \in SD_{8n} \setminus \langle a \rangle$.}
\end{array}\right.\]

\textbf{Case 2:} $n$ is odd. For $x \in Z(SD_{8n})$, by the proof of Theorem \ref{detour-ecentricity} \textbf{Case 1}, we have  $ecc_D(x) = 4n +11$  so $dds_D(x) = (1,\underbrace{0, 0, \ldots, 0}_\text{$4n + 6$}, 3, 0, 0, 0, 8n-4)$. For $x \in SD_{8n} \setminus Z(SD_{8n})$, again by the proof of Theorem \ref{detour-ecentricity} (\textbf{Case 2}), we have  $ecc_D(x) = 4n + 15$. Consequently, \[dds_D(x) = \left\{ \begin{array}{ll} 
(1,\underbrace{0, 0, \ldots, 0}_\text{$4n + 10$},4n - 1,0,0, 0, 4n) & \text{if $x \in \langle a \rangle \setminus Z(SD_{8n})$}\\

(1,\underbrace{0, 0, \ldots, 0}_\text{$4n + 10$},7, 0,0, 8n-8) & \text{if $x \in SD_{8n} \setminus \langle a \rangle$}
\end{array}\right.\]
\end{proof}

\section{Acknowledgement}
The first author wishes to acknowledge the support of MATRICS Grant  (MTR/2018/000779) funded by SERB, India.

\end{document}